\documentclass[reqno]{amsart}
\usepackage[left=1in,right=1in,top=1in,bottom=1in]{geometry}
\usepackage{tikz}
\usetikzlibrary{shapes,decorations,calc,arrows}
\usepackage[foot]{amsaddr}
\usepackage{amsthm}
\usepackage{amscd}
\usepackage{amsfonts}
\usepackage{amsmath}
\usepackage{amssymb}
\usepackage{mathrsfs}
\usepackage{multirow}
\usepackage{verbatim}
\usepackage{url}
\usepackage[hidelinks]{hyperref}
\usepackage{graphicx}
\usepackage{cite}
\usepackage{fancyhdr}
\usepackage{yfonts}
\usepackage{setspace}
\usepackage{titlesec}
\usepackage{enumitem}

\usepackage{ dsfont }

\usepackage{tocloft} 

\titleformat{\section}[hang]
{\normalfont\Large\bfseries}
{\thesection.}{0.5em}{}

\titlespacing*{\section}{0pc}{2pc}{0.25pc}

\titleformat{\subsection}[runin]
{\normalfont\large\bfseries}
{\thesubsection}{0.5em}{}

\titlespacing{\subsection}{0pc}{1.5pc}{0.5pc}

\newcommand{\N}{\mathbb{N}}

\newcommand{\C}{\mathbb{C}}

\newcommand{\<}{\left\langle}
\renewcommand{\>}{\right\rangle}

\newcommand{\dom}{\operatorname{dom}}

\newcommand{\cross}{\rotatebox[origin=c]{180}{\textnormal{\tiny\dag}}}

\newcommand{\E}{\mathcal{E}}

\newcommand{\co}{\text{co}}
\newcommand{\U}{\mathcal{U}}

\newtheorem{thm}{Theorem}
\newtheorem{prop}[thm]{Proposition}
\newtheorem{lem}[thm]{Lemma}
\newtheorem{cor}[thm]{Corollary}

\theoremstyle{definition}
\newtheorem{rem}[thm]{Remark}

\title{\textbf{On the genericity of irreducible subfactors}}
\author{Yoonkyeong Lee$^\circ$}
\address{$^\circ$Department of Mathematics, Michigan State University\hfill \url{leeyoo16@msu.edu}}
\author{Brent Nelson$^\bullet$}
\address{$^\bullet$Department of Mathematics, Michigan State University \hfill \url{brent@math.msu.edu}}
\date{}

\begin{document}


\begin{abstract}
We show that finitely generated irreducible $\mathrm{II}_1$ subfactors are generic in the following sense. Given a separable $\mathrm{II}_1$ factor $M$ and an integer $n\geq 2$, equip the set of $n$-tuples of self-adjoint operators in $M$ with norm at most $1$ with the metric $d(x,y) = \max_{1\leq i \leq n} \|x_i - y_i\|_2$. Then the set of $n$-tuples that generate an irreducible subfactor of $M$ forms a dense $G_\delta$ set in this metric space. On the way to proving this result, we show that closable derivations vanish on the anticoarse space associated to their kernels, which leads to new applications of conjugate systems in free probability.
\end{abstract}

\maketitle

\section*{Introduction}

Recall that a unital inclusion of von Neumann algebras $N\leq M$ is said to be \emph{irreducible} if $N'\cap M=\C$. In this case, both $N$ and $M$ (as well as any intermediate von Neumann algebra $N\leq P\leq M$) are necessarily factors. Irreducibility can be interpreted as saying that $N$ is influential to the structure of $M$; for example, it is equivalent to saying that $L^2(M)$ has no non-trivial $N$-$M$-subbimodules. Our main result shows that finite tuples generating irreducible subalgebras are generic.

\begin{thm}\label{thm}
Let $M$ be a separable $\mathrm{II}_1$ factor. For any integer $n\geq 2$, equip $(M_{s.a.})_1^{\oplus n}$ with the metric
    \[
        d(x,y) = \max_{1\leq i \leq n} \| x_i - y_i \|_2.
    \]
Then $\{x\in (M_{s.a.})_1^{\oplus n} \colon W^*(x)\leq M \text{ is irreducible}\}$ is a dense $G_\delta$ set in this metric space.
\end{thm}

\noindent Here and throughout we denote $M_{s.a.}:=\{x\in M\colon x=x^*\}$, $(M)_1:=\{x\in M\colon \|x\|\leq 1\}$, and $(M_{s.a.})_1:= M_{s.a.}\cap (M)_1$. The above theorem answers a question posed by S. Kunnawalkam Elayavalli, who along with coauthors D. Gao, G. Patchell, and H. Tan showed in \cite[Proposition 3.29]{GKEPT24} that finite tuples in a $\mathrm{II}_1$ factor generically generate a subfactor. Our proof (which can be found in Section~\ref{sec:proof_of_thm}) is essentially a refinement of theirs, where density is achieved through a Baire category argument. But a crucial difference is in the use of the following technical result.

\begin{prop}\label{prop}
Let $(M,\tau)$ be a tracial von Neumann algebra and suppose
    \[
        \delta\colon \dom(\delta)\to [L^2(M,\tau)\otimes L^2(M,\tau)]^{\oplus \infty}
    \]
is a derivation whose domain $\dom(\delta)\subset M$ is a unital $*$-subalgebra. Let $Q,N\leq M$ be the von Neumann subalgebras generated by $\ker{\delta}\cap \ker{\delta^*}$ and $\dom(\delta)$, respectively. If $\delta$ is closable as a densely defined operator on $L^2(N,\tau)$, then the anticoarse space for $Q\leq N$ satisfies
    \[
        L^2_{\cross}(Q\leq N, \tau) \subset \ker{\bar{\delta}}
    \]
where $\bar{\delta}$ denotes the closure of $\delta$.
\end{prop}

\noindent The anticoarse space in the above proposition is defined as
    \[
        L^2_{\cross}(Q\leq N, \tau):= \bigcap \left\{ \ker{T}\colon T\in \text{Hom}_{Q-Q}(L^2(N,\tau), L^2(Q,\tau)\otimes L^2(Q,\tau)\right\},
    \]
and forms a closed $Q$-subbimodule of $L^2(N,\tau)$ (see \cite[Definition 1.1]{Hay18} and \cite[Section 2.1]{HJKE24}). 
The utility of Proposition~\ref{prop} comes from the absorption properties of this anticoarse space, and relatively mild assumptions about the inclusion $Q\leq N$ can even force $L^2_{\cross}(Q\leq N,\tau)= L^2(N,\tau)$ so that $\bar\delta \equiv 0$. This holds, for example, if $\dim_Q L^2(N,\tau) < \infty$ (see \cite[Proposition 2.2]{HJKE24}), or if $Q$ is diffuse and the inclusion is regular in the sense that $\mathcal{N}_N(Q)''=N$ (see \cite[Proposition 1.2]{Hay18}). To put it another way, the proposition tells us that $\delta$ being non-zero precludes the inclusion from having certain properties. The proof of Proposition~\ref{prop} (which can be found in Section~\ref{sec:proof_of_prop}) relies on a standard technique that mollifies the derivation into a bounded operator that is bilinear over its kernel (see, for example, \cite{Pet09,Dab10,Nel17,Lee24}).

The anticoarse space has been a vital tool in applications of $1$-bounded entropy to the structure of von Neumann algebras (see, for example, \cite{Hay18, HJKE24}), and the previous proposition appears to open the door to similar applications in the non-microstates setting. For example, if a tuple $(x_i)_{i\in I}$ of operators admits a conjugate system, then the corresponding free difference quotients are closable by \cite[Corollary 4.2]{Voi98} and have diffuse kernels by \cite[Proposition 8.18]{MS17}. In Corollary~\ref{cor}, these facts are leveraged via Proposition~\ref{prop} in order to obtain structural properties of the inclusion $W^*(x_j\colon j\in J)\leq W^*(x_i\colon i\in I)$ when $J\subsetneq I$ is a non-empty proper subset. In particular, the inclusion is irreducible provided $|J|\geq 2$, and this is ultimately what is used in the proof of Theorem~\ref{thm}. It should also be noted that an important application of Corollary~\ref{cor} is to finite $q$-deformed semicircular families, which are known to admit conjugate systems by \cite{MS23}.

\subsection*{Acknowledgments}

We thank Srivatsav Kunnawalkam Elayavalli for several fruitful discusssions as well as suggesting the question in the first place; we thank Ben Hayes for providing helpful comparisons with microstates free entropy theory; and we thank David Jekel for his extensive feedback on the article as well the use of his observation in Remark~\ref{rem}.  Both authors were supported by NSF grant DMS-2247047.

\section{Proof of Proposition~\ref{prop}}\label{sec:proof_of_prop}

Define $\delta^\sharp(x) := \delta(x^*)^\dagger$ for $x\in \dom(\delta)$, where $\dagger$ is the conjugate linear isometry on $[L^2(M,\tau)\otimes L^2(M,\tau)]^{\oplus\infty}$ defined entrywise by $\xi\otimes \eta\mapsto (J_\tau \eta) \otimes (J_\tau\xi)$. Then
    \[
        \text{Re}(\delta):= \frac12(\delta+ \delta^\sharp) \qquad \text{ and } \qquad \text{Im}(\delta):= \frac{1}{2i}(\delta - \delta^\sharp)
    \]
are \emph{real derivations} in the sense that $\text{Re}(\delta)(x^*) = \text{Re}(\delta)(x)^\dagger$ and similarly for $\text{Im}(\delta)$. These derivations also satisfy $\delta = \text{Re}(\delta)+ i \text{Im}(\delta)$ and
    \[
        \ker{\delta} \cap \ker{\delta}^* = \ker{\text{Re}(\delta)} \cap \ker{\text{Im}(\delta)}.
    \]
So replacing $\delta$ with $\text{Re}(\delta)\oplus \text{Im}(\delta)$, we may assume $\delta$ is a real derivation with $\ker{\delta}\cap \ker{\delta}^*=\ker{\delta}$. Set $Q=\ker{\delta}''$ and note that $Q\subset \ker{\bar{\delta}}$. Also set $N:=\dom(\delta)''$.

Consider the unbounded operator $\Delta=\delta^*\bar\delta$, and for each $\alpha>0$ define $\zeta_\alpha\in B(L^2(N,\tau))$ by
    \[
        \zeta_\alpha:= (\frac{\alpha}{\alpha+\Delta})^{1/2},
    \]
which we note satisfies $\text{ran}(\zeta_\alpha)\subset \dom(\bar\delta)$ (see \cite[Section 2]{Pet09}). We also have that $\zeta_\alpha \to 1$ in the strong operator topology as $\alpha\to\infty$. To see this, it suffices by \cite[Theorem \rm{II}.4.7]{Tak02} to show $\zeta_\alpha^2\to 1$ in the strong operator topology as $\alpha\to \infty$, and this is clear since $\zeta_\alpha^2(\xi)-\xi=\frac{\alpha}{\alpha+\Delta}(\xi)-\xi=\frac{\Delta}{\alpha+\Delta}(\xi)\to 0$ as $\alpha\to\infty$. Additionally, we claim that $\zeta_\alpha$ is $Q$-bilinear. Recall that $\bar{\delta}$ restricts to a derivation on $N\cap\dom(\bar{\delta})$ (see \cite{DL92} and \cite[Section 2]{Pet09}), and so for $x\in Q\subset \ker{\bar\delta}$, $\xi\in \dom(\Delta)$, and $y\in \dom(\delta)$ one has
    \begin{align*}
        \< \Delta(x\xi),y\>_{\tau} &=\<\bar{\delta}(x\xi),\delta(y)\>_{\tau\otimes \tau}=\<x\bar{\delta}(\xi),\delta(y)\>_{\tau\otimes \tau} \\
        &=\< \bar{\delta}(\xi),x^*\delta(y)\>_{\tau\otimes \tau} =\<\bar{\delta}(\xi),\bar{\delta}(x^*y)\>_{\tau\otimes\tau} = \<\Delta(\xi), x^* y\>_{\tau}=\<x\Delta(\xi),y\>_{\tau}.
    \end{align*}
Thus $\Delta(x\xi)=x\Delta(\xi)$. Now for $\xi\in L^2(N,\tau)$, denote  $f:=\zeta_\alpha^2(x \xi)=\frac{\alpha}{\alpha+\Delta}(x \xi)$ and $g:=\zeta_\alpha^2(\xi)=\frac{\alpha}{\alpha+\Delta}(\xi)$. Then for $x\in Q$ one has
    \begin{align*}
        \frac{1}{\alpha}(\alpha+\Delta)(f-x g)&=x \xi-x g- \frac{1}{\alpha}\Delta(x g)\\
        &=x \xi-x g-x\frac{1}{\alpha}\Delta(g)=x \xi-x \left(1+\frac{\Delta}{\alpha}\right)(g)=x \xi-x \xi=0.
    \end{align*}
Thus $\zeta_\alpha^2(x \xi)=x \zeta_\alpha^2(\xi)$ and the functional calculus implies $\zeta_\alpha(x \xi)=x\zeta_\alpha(\xi)$. Similarly, $\zeta_\alpha(\xi x )=\zeta_\alpha(\xi)x$.

Now, for each $\alpha>0$, the above discussion implies
    \[
        T_\alpha:= \bar{\delta}\circ \zeta_\alpha
    \]
is a bounded, $Q$-bilinear map. Since $L^2(M,\tau)$ embeds into $L^2(Q,\tau)^{\oplus\infty}$ as either a left or right $Q$-module (though not necessarily as a $Q$-bimodule), we may view $T_\alpha$ as a bounded, $Q$-bilinear map from $L^2(N,\tau)$ to $[L^2(Q,\tau)\otimes L^2(Q,\tau)]^{\oplus\infty}$. Thus $T_\alpha(\xi)=0$ for all $\xi \in L^2_\dagger(Q\leq N,\tau)$, and since $\zeta_\alpha(\xi)\to \xi$ as $\alpha\to \infty$ it follows that $\xi\in \dom(\bar{\delta})$ with $\bar{\delta}(\xi) =0$.$\hfill\blacksquare$\\

We conclude this section with an application of Proposition~\ref{prop} that will be needed in the proof of Theorem~\ref{thm}, but which may be of independent interest.

\begin{cor}\label{cor}
Let $(M,\tau)$ be a tracial von Neumann algebra and suppose $(x_i)_{i\in I} \subset M_{s.a.}$ is a tuple admitting a conjugate system. Then for any non-empty proper subset $J\subsetneq I$ the inclusion
    \[
        W^*(x_j\colon j\in J) \leq W^*(x_i\colon i\in I)
    \]
is infinite index and non-regular. If $A\leq W^*(x_i\colon i\in I)$ is abelian with diffuse intersection $A\cap W^*(x_j\colon j\in J)$, then $A$ is not regular. Furthermore, if $|J|\geq 2$, then the inclusion is also irreducible.
\end{cor}
\begin{proof}
Denote $N:=W^*(x_i\colon i\in I)$ and  $Q:=W^*(x_j\colon j\in J)$. For each $i\in I$, let
    \[
        \partial_i\colon \C\<x_i\colon i\in I\> \to L^2(M,\tau)\otimes L^2(M,\tau)
    \]
be the \emph{free difference quotient}: the derivation defined by $\partial_i(x_j) = \delta_{i=j} 1\otimes 1$. The free difference quotients are real derivations with $\ker{\partial_i}\supset \C\<x_k\colon k\in I\setminus\{i\}\>$, and the existence of a conjugate system for $(x_i)_{i\in I}$ asserts that each $\partial_i$ is closable (see \cite[Corollary 4.2]{Voi98}). Thus for $i\in I\setminus J$ we have
    \[
        L^2_{\cross}(Q\leq N, \tau) \subset L^2_{\cross}( W^*(\ker{\partial_i}) \leq N, \tau) \subset \ker{\bar\partial_i}
    \]
by Proposition~\ref{prop}, where the first inclusion follows from the fact that $L^2(W^*(\ker{\partial_i}),\tau)\otimes L^2(W^*(\ker{\partial_i}),\tau)$ can be embedded as a $Q$-bimodule in $[L^2(Q,\tau)\otimes L^2(Q,\tau)]^{\oplus \infty}$. Since $\partial_i$ is non-zero, it follows from \cite[Proposition 2.2]{HJKE24} that the inclusion $Q\leq N$ has infinite index, and since $Q$ is diffuse by \cite[Proposition 8.18]{MS17} we further have that the inclusion is non-regular by \cite[Proposition 1.2]{Hay18}.

Now, suppose $A\leq N$ is abelian with $A\cap W^*(x_j\colon j\in J)$ diffuse. Noting that $A \subset \mathcal{N}_N(A\cap W^*(x_j\colon j\in J))''$, it follows from \cite[Proposition 1.2]{Hay18} that
    \[
        A \subset L^2_{\cross}(A\cap W^*(x_j\colon j\in J) \leq N,\tau) \subset L^2_{\cross}( W^*(\ker{\partial_i}) \leq N, \tau) \subset \ker{\bar\partial_i}.
    \]
Let $\delta$ be the restriction of $\bar\partial_i$ to $N\cap \dom(\bar\partial_i)$, which is still a closable derivation with $\bar\delta = \bar\partial_i$ (see \cite{DL92} and \cite[Section 2]{Pet09}) and $A\subset \ker{\delta}$. Applying Proposition~\ref{prop} to $\delta$ we obtain
    \[
        L^2_{\cross}(A\leq N,\tau) \leq L^2_{\cross}( W^*(\ker{\delta}) \leq N,\tau) \subset \ker{\delta}.
    \]
Noting that $A$ is diffuse because it contains the diffuse algebra $A\cap W^*(x_j\colon j\in J)$, we see that the regularity of $A$ would contradict $\partial_i$ being non-zero by \cite[Proposition 1.2]{Hay18}.

For irreducibility of $Q\leq N$, we more or less argue as in \cite[Remark 3]{Dab10}. Assume $|J|\geq 2$ and fix distinct $j,k\in J$. Fixing $z\in Q'\cap N$, we have that $B:=W^*(x_k)$ is diffuse and lies in $Q$. Since $z\in B'\cap N\subset \mathcal{N}_N(B)''$, \cite[Proposition 1.2]{Hay18} and Proposition~\ref{prop} gives us
    \[
        z\in L^2_{\cross}(B\leq N,\tau) \subset L^2_{\cross}(W^*(\ker{\partial_j})\leq N,\tau) \subset \ker{\bar\partial_j},
    \]
where the first inclusion is argued as above. But then
    \[
        z\otimes 1 = \bar\partial_j(zx_j) = \bar\partial_j(x_jz) = 1\otimes z,
    \]
and so applying $1\otimes \tau$ to the above gives $z=\tau(z)\in \C$. Hence $Q\leq N$ is irreducible.
\end{proof}

\section{Proof of Theorem~\ref{thm}}\label{sec:proof_of_thm}

The following lemma is likely well-known to experts. 

\begin{lem}\label{lem}
For a tracial von Neumann algebra $(M,\tau)$ with von Neumann subalgebra $N$, the following are equivalent:
    \begin{enumerate}[label=(\roman*)]
        \item $N\leq M$ is irreducible;
        \item for all $x\in M_{s.a.}$ and $\epsilon>0$ there exists a unitary $u\in N$ so that
            \[
                \| x - \tau(x) \|_2 < \| [x,u] \|_2 + \epsilon;
            \]

        \item for all $x\in M_{s.a.}$ and $\epsilon>0$ there exists $p\in (N)_1$ so that
            \[
                \| x- \tau(x) \|_2 < \| [x,p]\|_2 + \epsilon.
            \]
    \end{enumerate}
\end{lem}
\begin{proof}
\textbf{(i)$\Rightarrow$(ii):} Since $N'\cap M=\C$, the conditional expectation $\E_{N'\cap M}$ onto $N'\cap M$ equals $\tau$, and so $\tau(x)=\E_{N'\cap M}(x)\in \overline{\co}^{\|\cdot\|_2}\{uxu^\ast:u\in \U(N)\}$. Given $\epsilon >0$, let $\alpha_1,...,\alpha_n\in (0,1)$ and $u_1,...u_n\in \U(N)$ be such that $\sum_{i=1}^n\alpha_i=1$ and $\|\sum_{i=1}^n \alpha_iu_ixu_i^\ast-\tau(x)\|_2<\epsilon$. Then we claim there exists an $1\leq i \leq n$ such that $\|u_ixu_i^\ast-x\|_2 > \|x-\tau(x)\|_2-\epsilon$. If not, then the reverse triangle inequality gives the contradiction
        \begin{align*}
            \left\| \sum_{i=1}^n \alpha_iu_ixu_i^*-\tau(x) \right\|_2&\geq  \|x-\tau(x)\|_2 - \sum_{i=1}^n\alpha_i\|u_ixu_i^\ast -x\|_2  \\
            & \geq  \|x-\tau(x)\|_2-(\|x-\tau(x)\|_2-\epsilon )=\epsilon.
        \end{align*}  
    Therefore, $ \| x - \tau(x) \|_2 < \| [x,u_i] \|_2 + \epsilon$ for some $1\leq i \leq n$.\\
    
    \noindent\textbf{(ii)$\Rightarrow$(iii):} Take $p=u$.\\

    \noindent\textbf{(iii)$\Rightarrow$(i):} Let $x\in N'\cap M$. Since $\| [x,p]\|_2=0$ and $\epsilon>0$ was arbitrary, we have $\|x-\tau(x)\|_2=0$.
\end{proof}

We now fix a separable $\mathrm{II}_1$ factor $M$ with $\{a_k\in M_{s.a.}\colon k\in \N\}$ a $2$-norm dense subset of $M_{s.a.}$. We also fix an integer $n\geq 2$ and equip $(M_{s.a.})_1^{\oplus n}$ with the metric
    \[
        d( x, y) := \max_{1\leq i \leq n}\| x_i - y_i\|_2.
    \]
Recall that $(M_{s.a.})_1^{\oplus n}$ is complete with respect to this metric (see \cite[Proposition III.5.3]{Tak02}).

For each $k,m\in \N$, define
    \[
        G_{k,m}:=\left\{ x\in (M_{s.a.})_1^{\oplus n}\colon \exists p\in (W^*(x))_1 \text{ such that } \|a_k - \tau(a_k)\|_2 < \|[a_k, p]\|_2 + \frac1m \right\}.
    \]
We first observe that $G_{k,m}$ is open with respect to $d$. Given $x\in G_{k,m}$, let $p\in (W^*(x))_1$ be the element witnessing its inclusion in this set. By \cite[Lemma 2.2]{JP24}, we can then approximate $p$ well enough in $2$-norm by a $*$-polynomial $q(x)\in \C\<x\>$ so that $\| a_k - \tau(a_k)\|_2 < \| [a_k, q(x)]\|_2+\frac1m$ and so that $\|q(y)\|\leq 1$ for any tuple $y\in (M_{s.a.})_1^{\oplus n}$. Since $\|q(x) - q(y)\|_2 \leq C d(x,y)$ for some constant depending only $q$, it then follows $G_{k,m}$ is open.

We next claim that the $G_{k,m}$ are dense in $(M_{s.a.})_1^{\oplus n}$ with respect to $d$. Fix $k,m\in \N$ and $x=(x_1,\ldots, x_n) \in (M_{s.a.})_1^{\oplus n}$, and set $x_0:=a_k$. Let $\mathcal{U}$ be a free ultrafilter on $\N$ and let $(M^{\mathcal{U}},\tau^{\mathcal{U}})$ be the corresponding ultrapower. We iteratively apply \cite[Theorem 0.1]{Pop14} to $\{x_j\colon 0\leq j \leq n\}\subset M\subset M^{\mathcal{U}}$ in order to find a free semicircular family $s_0,s_1,\ldots, s_n\in  M^{\mathcal{U}}$ that is free from $\{x_j\colon 0\leq j\leq n\}$. For each $j=0,1\ldots, n$ and $t>0$, denote
    \[
        x_j(t):= x_j + t s_j,
    \]
and denote $x(t):=(x_0(t),\ldots, x_n(t))$. Then $x(t)$ admits a conjugate system by \cite[Propsotion 3.7]{Voi98} for any $t>0$, and using $n\geq 2$ it follows from Corollary~\ref{cor} that $W^*(x_i(t)\colon 1\leq i \leq n) \leq W^*(x_j(t)\colon 0\leq j\leq n)$ is irreducible. So invoking Lemma~\ref{lem} there exists a unitary $u(t)\in W^*(x_i(t)\colon 1\leq i \leq n)$ so that
    \[
        \| x_0(t) - \tau^{\mathcal{U}}(x_0(t)) \|_2 < \| [x_0(t), u(t)] \|_2 + \frac{1}{2m}.
    \]
Since $\| x_0(t) - a_k \|_2 = t$ and $\tau^{\mathcal{U}}(x_0(t) - a_k) = \tau^{\mathcal{U}}(ts_0)=0$, for $t\leq \frac{1}{6m}$ we have
    \[
        \| a_k - \tau(a_k)\|_2 < \| [a_k, u(t)]\|_2 + \frac1m.
    \]
Using \cite[Lemma 2.2]{JP24} again we can replace $u(t)$ in the above inequality with a $*$-polynomial $q(x_1(t),\ldots, x_n(t))$ satisfying $\|q( y_1,\ldots, y_n)\|\leq 1$ for any $n$-tuple $(y_1,\ldots, y_n)$ with $\|y_i \| \leq 1+t$. Now, for each $i=1,\ldots, n$ lift $s_i\in M^{\mathcal{U}}$ to a sequence $(s_i^{(\ell)})_{\ell\in \N} \subset (M_{s.a.})_1$. It follows that $( x_i + t s_{i}^{(\ell)})_{\ell\in \N}$ is a lift for $x_i(t)$ with $\| x_i + t s_i^{(\ell)}\|\leq 1+t$ for all $\ell\in \N$. Consequently, $\|q(x_1+t s_1^{(\ell)},\ldots,  x_n+t s_n^{(\ell)} )\| \leq 1$ for all $\ell\in \N$ and
    \[
        \lim_{\ell\to\mathcal{U}} \left\| \left[a_k , q(x_1+t s_1^{(\ell)},\ldots,  x_n+t s_n^{(\ell)} ) \right] \right\|_2 = \| \left[a_k, q(x_1(t),\ldots, x_n(t))\right] \|_2  > \| a_k - \tau(a_k)\|_2 - \frac1m.
    \]
Thus there exists $A\in \mathcal{U}$ so that
    \[
        \|a_k - \tau(a_k)\|_2 < \left\| \left[ a_k, q(x_1+t s_1^{(\ell)},\ldots,  x_n+t s_n^{(\ell)} ) \right] \right\|_2 +\frac1m
    \]
for any $\ell\in A$. Fix any such $\ell\in A$ and set $y:= \frac{1}{1+t}(x_1+t s_1^{(\ell)},\ldots, x_n + t s_n^{(\ell)})_{i\in I} \in (M_{s.a.})_1^{\oplus n}$. Then the above implies $y\in G_{k,m}$ with
    \[
        d(x, y) = \max_{1\leq i \leq n} \frac{t}{1+t}\|x_i - s_i^{(\ell)}\|_2 \leq \frac{2t}{1+t}.
    \]
Thus $G_{k,m}$ is dense.

Now, the Baire category theorem implies
    \[
        G:=\bigcap_{k,m\in \N} G_{k,m}
    \]
is a dense $G_\delta$ set in $(M_{s.a.})_1^{\oplus n}$ with respect to $d$, and we claim $W^*(x) \leq M$ is irreducible for all $x\in G$. Indeed, for $a\in M_{s.a.}$ and $\epsilon>0$, let $k,m\in \N$ be such that $\|a - a_k\|_2 \leq \frac{\epsilon}{5}$ and $\frac1m \leq \frac{\epsilon}{5}$. Then $x\in G_{k,m}$ yields $p\in (W^*(x))_1$ so that
   \[
       \| a - \tau(a)\|_2 \leq \| a_k - \tau(a_k)\|_2 + \frac{2\epsilon}{5} < \|[a_k,p]\|_2 + \frac{3\epsilon}{5} \leq \| [a, p] \|_2 + \epsilon.
   \]
Thus $W^*(x)\leq M$ is irreducible by Lemma~\ref{lem}. $\hfill\blacksquare$

\begin{rem}\label{rem}
The above proof also works for countably infinite tuples $(M_{s.a.})_1^{\oplus \N}$ equipped with the metric
    \[
        d(x,y):=\sup_{i\in\N} \|x_i - y_i\|_2.
    \]
However, it was brought to our attention by D. Jekel that in this case one can obtain a stronger result via a shorter proof; namely, the set $\{x\in (M_{s.a.})_1^{\oplus \N}\colon W^*(x) = M\}$ is a dense $G_\delta$ set in this metric space. It is $G_\delta$ since for a $2$-norm dense subset $\{a_k\in M_{s.a.}\colon k\in \N\}$ of $M_{s.a.}$ one has
    \[
        \{x\in (M_{s.a.})_1^{\oplus \N}\colon W^*(x) = M\} = \bigcap_{k,m\in \N} \left\{ x\in (M_{s.a.})_1^{\oplus \N} \colon \exists p\in (W^*(x))_1 \text{ such that } \|p- a_k\|_2 < \frac{1}{m}\right\},
    \]
and the sets indexed by $k,m\in \N$ are open by an argument similar to the one appearing in the proof of Theorem~\ref{thm}. To see the density of this set, fix $x=(x_i)_{i\in \N} \in (M_{s.a.})_1^{\oplus \N}$ and $\epsilon>0$. Apply Theorem~\ref{thm} (or \cite[Proposition 3.29]{GKEPT24}) to $(x_1,x_2)$ in order to find $(y_1,y_2)\in (M_{s.a.})_1^{\oplus 2}$ satisfying $\|x_i - y_i\|_2 < \epsilon$ for $i=1,2$ with $W^*(y_1,y_2)$ a $\mathrm{II}_1$ factor. Then fix a non-zero projection $q\in W^*(y_1,y_2)$ with $\|q\|_2 < \frac{\epsilon}{4}$, and let $\{b_i\in ((qMq)_{s.a.})_1 \colon i\in \N\}$ be a countable generating set for $qMq$. Setting
    \[
        y_i:= (1-q)x_i(1-q) + b_{i-2}
    \]
for each $i\geq 3$, it follows that $\|y_i\| \leq 1$ and
    \[
        \| x_i - y_i \|_2 \leq \| q x_i (1-q) + (1-q) x_i q + q x_i q \|_2 + \| q b_{i-2} q\|_2 \leq 4 \|q\|_2 < \epsilon.
    \]
Setting $y:=(y_i)_{i\in \N} \in (M_{s.a.})_1^{\oplus \N}$, we therefore have $d(x,y) <\epsilon$, and it remains to show $W^*(y)=M$. The factoriality of $W^*(y_1,y_2)$ yields partial isometries $v_1,\ldots, v_\ell\in W^*(y_1,y_2)$ satisfying $v_s ^* v_s \leq q$ and $\sum_{s=1}^\ell v_sv_s^* =1$. Consequently, for any $a\in M$ one has
    \[
        a= \sum_{s,t=1}^\ell v_s v_s^* a v_t v_t^* = \sum_{s,t=1}^\ell v_s q( v_s^* a v_t) q v_t^*.
    \]
Since $q\in W^*(y_1,y_2)\subset W^*(y)$, it follows that $q y_i q = b_{i-2}\in W^*(y)$ for all $i\geq 3$ and therefore $pMp \subset W^*(y)$. Thus the above expression lies in $W^*(y)$ and therefore $W^*(y)=M$.
\end{rem}

\bibliographystyle{amsalpha}
\bibliography{reference}

\end{document}